\documentclass{conm-p-l}

\usepackage{graphicx, amsthm, latexsym, amssymb, tikz, xcolor}

\newtheorem{theorem}{Theorem}[section]
\newtheorem{lemma}[theorem]{Lemma}
\newtheorem{cor}[theorem]{Corollary}
\newtheorem{prop}[theorem]{Proposition}
\newtheorem*{conjecture}{Conjecture}

\theoremstyle{definition}

\theoremstyle{remark}

\numberwithin{equation}{section}

\newcommand{\sumi}{\sum\limits_{i=1}^{n}}

\newcommand{\ceil}[1]{\left \lceil #1 \right \rceil }

\newcommand{\R}{\mathbb{R}}

\newcommand{\ignore}[1]{}

\newcommand{\mylabel}[5]{\put(#1,#2){#5}\put(#3,#4){#5}}

\newcommand{\mysetb}[7]{\color{#7}\put(#1,#2){\line(1,0){#3}}\put(#4,#5){\line(1,0){#6}}\color{black}}

\definecolor{darkred}{rgb}{0.8,0,0}
\definecolor{darkblue}{rgb}{0,0,0.8}
\definecolor{darkyellow}{rgb}{0.9,0.84,0}
\definecolor{darkgreen}{rgb}{0,0.8,0}

\begin{document}

\title
{Double-Interval Societies}

\author[M. Klawe] {Maria Klawe}%
\address{President, Harvey Mudd College}%
\email{klawe@hmc.edu}%

\author[K.L. Nyman] {Kathryn L. Nyman}%
\address{Department of Mathematics, Willamette University}%
\email{knyman@willamette.edu}%

\author[J.N. Scott] {Jacob N. Scott}%
\address{Department of Mathematics, University of California, Berkeley}%
\email{jnscott@math.berkeley.edu}%

\author[F.E. Su]{Francis Edward Su*}%
\address{Department of Mathematics, Harvey Mudd College}%
\email{su@math.hmc.edu}%
\thanks{*Work partially supported by NSF Grant DMS-1002938}

\subjclass[2000]{Primary 52A35; Secondary 91B12}

\date{}

\begin{abstract}
Consider a society of voters, each of whom specify an \emph{approval set} over a linear political spectrum.  
We examine \emph{double-interval societies}, in which each person's approval set is represented by two disjoint closed intervals, and study this situation where the approval sets are \emph{pairwise-intersecting}: every pair of voters has a point in the intersection of their approval sets.  The \emph{approval ratio} for a society is, loosely speaking, the popularity of the most popular position on the spectrum.  We study the question: what is the minimal guaranteed approval ratio for such a society?  We provide a lower bound for the approval ratio, and examine a family of societies that have rather low approval ratios.  These societies arise from \emph{double-$n$ strings}: arrangements of $n$ symbols in which each symbol appears exactly twice. 
\end{abstract}

\maketitle
\section{Introduction}
Consider the voting model of Berg et.\,al.\cite{NSTW}\ in which a political spectrum $X$ is viewed as a continuum, with liberal positions on the left and conservative positions on the right, and in which each voter $v$ ``approves'' an interval of positions along this line.  For example, a tolerant moderate might approve a wide interval near the middle of the line, while an intolerant partisan may approve a narrower interval near one of the ends.

More formally, a \emph{society} is a spectrum $X$ together with a set of voters $V$ and a 
collection of approval sets $\{ A_{v} \}$, one for each voter.  A point on the spectrum $X$ is called a \emph{platform}.
In our situation, we imagine $X$ to be $\R$, and each approval set $A_{v}$ is a closed interval that represents the set of all platforms that $v$ approves.

Now suppose that every pair of people can agree on some platform; that is, their intervals overlap.  In this situation, Helly's Theorem \cite{helly} implies that there exists a point on the line that lies in everyone's approval set, i.e., there is a platform that everyone approves.  Thus a strong hypothesis (pairwise intersecting sets) produces a strong conclusion (a point in all the sets).  However, in voting theory, we are usually not looking for unanimity, but may be satisfied with a platform that has high \emph{approval ratio}: the fraction of voters that approve this platform.

Various authors have relaxed the hypotheses.  Berg et.\,al.\cite{NSTW}\ define a linear $(k,m)$-agreeable society in which voter preferences again are modeled by closed intervals in $\R^1$.  In this society, given any set of $m$ voters, there exists a subset of $k$ voters whose approval intervals mutually intersect.   They prove that there must exist some platform with approval ratio $\frac{k-1}{m-1}$.
Another generalization by Hardin \cite{hardin} looks at approval intervals on a circle rather than a line, and finds that with $(k,m)$-agreeability, the approval ratio of the society is at least $\frac{k-1}{m}$. 
  
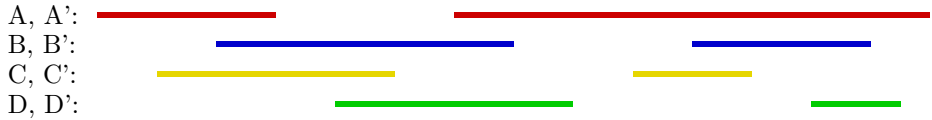
\begin{figure}[h]
\centering
\setlength{\unitlength}{22.5pt}
\begin{picture}(15,2)

\linethickness{2pt}
\color{darkred}
\put(1,1.8){\line(1,0){3}}
\put(7,1.8){\line(1,0){8}}
\color{darkblue}
\put(3,1.3){\line(1,0){5}}
\put(11,1.3){\line(1,0){3}}
\color{darkyellow}
\put(2,.8){\line(1,0){4}}
\put(10,.8){\line(1,0){2}}
\color{darkgreen}
\put(5,0.3){\line(1,0){4}}
\put(13,0.3){\line(1,0){1.5}}

\color{black}
\put(-0.5,1.65){A, A':}
\put(-0.5,1.15){B, B':}
\put(-0.5,.65){C, C':}
\put(-0.5,0.15){D, D':}

\put(2,-1.5){}
\end{picture}
\caption{A pairwise-intersecting society of size 4 with approval number 3.}
\label{examplesociety}
\end{figure}

We generalize the one-interval model to a society in which every member is identified with two disjoint approval intervals and call such a society a {\it double-interval society}.  This situation may arise naturally in the context of voting to account for voters who do not place candidates along a linear spectrum in exactly the same order, or to account for voters who find disjoint sets of platforms appealing for entirely different reasons (e.g., for being a party purist, or having the ability to work across party lines).   In a scheduling context, such intervals might model a society of workers, each of whom has two different work shifts.

Figure \ref{examplesociety} illustrates a double-interval society with four voters.  The approval sets of each voter have been separated vertically so they are easier to see.  Note that the approval sets are pairwise-intersecting: each voter overlaps every other voter in one or both of their approval intervals.  In this example, there are several platforms approved by three voters, but no platform is approved by all four.  The {\it approval number} of a platform $a(p)$ is the number of voters (in a society $S$) who approve of platform $p$.  The {\it approval number} of a society $a(S)$ is the maximum approval number over all platforms in the spectrum $X$.  That is
$$ a(S) = \max_{p \in X} a(p).$$
Finally, define the {\it approval ratio} of a society
to be the approval number of $S$ divided by the number of voters in $S$.

The main question we address in this paper is: what is the minimal approval ratio of a pairwise-intersecting, double-interval society with $n$ voters?

Examples suggest that the minimal approval ratio of such societies is $1/3$; that is, there is always a platform that will get at least a third of the votes.  Our results in this paper attempt to clarify this intuition.

We will first examine a family of double-interval societies with low approval ratios that have regular patterns of interval overlap.
These arise from the construction of what we call \emph{double-$n$} strings, defined in the next Section.
The combinatorics of such strings are quite nifty and provide a lower bound for the approval ratio of societies in this family 
(Theorem \ref{klowerbound}) as well as an upper bound (Theorem \ref{klaweupperbound}) for societies in this family. 
Roughly speaking, the double-$n$ strings produce societies with asymptotic approval ratios between $0.348$ and $0.385$.  

We will also prove a general lower bound for the approval ratio of any pairwise-intersecting, double-interval society in Theorem \ref{lowerbound}, which shows the approval ratio is always greater than $0.268$.

Then we ask if we can find specific societies with lower approval ratios than the ones arising from double-$n$ strings, and discover that there are such examples.  We find them by modifying the construction that comes from double-$n$ strings.  See Table \ref{boundresults}. However, all of these examples have approval ratio greater than or equal to 1/3.

\section{Double-$n$ String Societies}

Double-interval societies with regular patterns of interval overlap can be represented by \emph{double-$n$ strings}, that is, strings of length $2n$ containing exactly two occurrences of each of $n$ symbols.  At times we will also represent double-$n$ strings as strings of the symbols $1, \ldots, n$.  
We define the \emph{distance} between two distinct symbols in a double-$n$ string to be the minimum distance between a pair of occurrences of the symbols, where the distance between two adjacent symbols is taken as 1. The \emph{diameter} of a double-$n$ string is the maximum over all $1\leq i<j\leq n$ of the distance between $i$ and $j$.   We will call two entries in the list \emph{adjacent} if their positions in the list differ by no more than the diameter of the string.  

Let $\delta(n)$ be the minimum diameter over all double-$n$ strings.

\begin{figure}[h]
\centering
\setlength{\unitlength}{12pt}
\begin{picture}(19,12)
\linethickness{2pt}

\mysetb{0}{10.8}{3}{11.2}{1.2}{3.1}{darkred}
\mysetb{1.4}{9.6}{3}{7}{4.8}{3.1}{darkblue}
\mysetb{2.8}{8.4}{3}{9.8}{2.4}{3.1}{darkyellow}
\mysetb{4.2}{7.2}{3}{12.6}{0.0}{3.1}{darkgreen}
\mysetb{5.6}{6.0}{3 } {8.4} {3.6} {3.1 } {orange}

\mylabel{1.4}{11}{12.6}{1.4}{A}
\mylabel{2.8}{9.8}{8.4}{5}{B}
\mylabel{4.2}{8.6}{11.2}{2.6}{C}
\mylabel{5.6}{7.4}{14}{0.2}{D}
\mylabel{7} {6.2 } {9.8} {3.8} {E}

\end{picture}
\caption{A society represented by the double-$5$ string $ABCDEBECAD$.}
\label{double-n}
\end{figure}
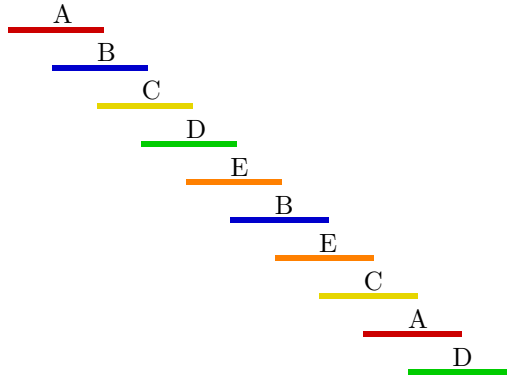

We can construct a pairwise-intersecting double-interval society from a double-$n$ string with diameter $d$ by assigning intervals of equal width to the symbols, long enough so that each interval overlaps the intervals of the $d$ symbols to its right and left.

For example, consider the double-5 string $ABCDEBECAD$. This string has diameter 2, since any pair of symbols $A$ through $E$ appear somewhere in this list separated by at most one other symbol (e.g., the second occurrences of $A$ and $E$ in this string are distance 2 apart). 
We build a society from this string by assigning intervals of equal width as in Figure \ref{double-n}.  
This society has approval number 3 as can be seen 
since the right endpoint of $A$'s first interval intersects the left endpoint of $C$'s first interval, and both intersect $B$'s first interval. Hence we see that $\delta(5) \leq 2$ (and in fact $\delta(5)=2$).  Note that in general the approval number of a 
society with an underlying double-$n$ string is one more than the diameter, that is, $a(S) = d+1$.  

\section{Asymptotic approval ratios for double-$n$ string societies}

If $S$ is arises from a double-$n$ string with diameter $d$, then since $a(S)=d+1$, we see that the minimal approval ratio of such a society is 
$(\delta(n)+1)/n$.  By taking limits, we see that 
$$
\Delta = \lim_{n \rightarrow \infty} \frac{\delta(n)+1}{n} =  \lim_{n \rightarrow \infty} \frac{\delta(n)}{n}
$$
is the asymptotic approval ratio for societies arising from double-$n$ strings.  In this section, we will show that 
$$8/23 \leq \Delta \leq   5/13.$$
It is clear that for $n > 1$ we have $\delta(n-1) \leq \delta(n)$ since for any double-$n$ string we can form a double-$(n-1)$ string of no larger diameter by deleting both occurrences of the $n$-th symbol. Given a double-$n$ string $S$ we label the symbols as $1,2, \ldots, n$ according to the left to right order of their first occurrence within $S$. 

It is easy to see that $\Delta \leq 1/2$ since the double-$n$ string $1,2,\ldots , n, 1, 2, \ldots, n$ shows that $\delta(n) \leq n/2$ for $n$ at least 2.  In fact, we can show $\delta(n) < n/2$ for $n$ at least 3. Although this does not change the upper bound on $\Delta$, we will need this result for the lower bound $8/23 \leq \Delta$. 

\begin{lemma} If $n>2$, then $\delta(n) < n/2$. 
\label{lem:n/2}
\end{lemma}

\begin{proof}
The case $\delta(3) = 1$ follows from the double-3 string 1,2,3,1,2,3 and the case $\delta(4) = 1$ follows from the double-4 string 1,2,3,4,1,3,2,4. The proof for the general case $n>4$ is based on this double-4 string. Let $r = \lfloor n/4 \rfloor$.
 We will partition 1,2,\ldots,$n$ into four strings $S_1$, $S_2$, $S_3$, $S_4$ where each is of length $ r$ or $r+1$ depending on the value of  $n \mathbin{\mathrm{mod}} 4$. We prove the result by looking at the diameter of the double-$n$ string $T(n) = S_1, S_2, S_3, S_4, S_1, S_3, S_2, S_4$. 

Suppose $n = 4r$ for some positive integer $r$.  We let $S_i$ be the string $(i-1)r+1, (i-1)r+2,\ldots, ir $ of length $r$  and it is easy to check that the diameter of $T(n)$ is $2r - 1$ and we have $2r -1  <  n/2$. For $n = 4r+1$, let $S_1$, $S_2$, $S_3$ be as before and let $S_4$ be the string $3r+1$,\ldots,$4r+1$ of length $r+1$. Now the diameter of $T(n)$ is $2r$, and again we have $2r < n/2$ as desired. For $n = 4r+2$ we set $S_1$ and $S_4$ to have length $r$ and $S_2$ and $S_3$ to have length $r+1$. Consider $b$ in $S_i$ and $c$ in $S_j$ with $i < j$. It is easy to see that unless $i = 2$ and $j = 3$, the distance between $b$ and $c$ in $T(n)$ is at most $2r$ since at least one of $S_i$ and $S_j$ has length $r$ and the other has length at most $r+1$. Moreover, for $b$ in $S_2$ and $c$ in $S_3$ the distance between $b$ and $c$ in $T(n)$ is at most $r+1$ since both of the substrings  $S_2$, $S_3$ and $S_3$, $S_2$ occur in $T(n)$. Thus the diameter of $T(n) = 2r$ and we have $2r < 2r+1 = n/2$. Finally for $n = 4r+3$ we set $S_1$ to have length $r$ and $S_2$, $S_3$, $S_4$ to have length $r+1$. In this case it is easy to see that the diameter of $T(n)$ is $2r+1$ and we have $2r+1 < 2r+3/2 = n/2$.
\end{proof}

For simplicity, without loss of generality assume that if the first occurrence of symbol $m$ occurs at position $i$ in a double-$n$ string, then all symbols at positions $1 \leq j < i$ are less than $m$ (otherwise this condition can be satisfied by a permutation of the symbols in the double-$n$ string).  From Lemma \ref{lem:n/2}, it is sufficient to consider double-$n$ strings that have diameter less than $n/2$. It is also easy to obtain the lower bound $\Delta \geq 1/3$ as shown in the following lemma.

\begin{lemma}
Let $r$ be a positive integer.  We have
$\delta(3r+1) \geq r$.
\label{lem:3r+1}
\end{lemma}
\begin{proof}
Let $n=3r+1$. In any double-$n$ string of diameter $d$, the first occurrence of the symbol 1 can be adjacent to at most $d$ other symbols while the second occurrence can be adjacent to at most $2d$.  Because 1 must be adjacent to all $n-1$ other symbols, $d+2d \geq n-1=3r$, and so $d \geq r$.\end{proof}

\begin{lemma}
\label{distinctness}
In a double-$n$ string with diameter $d$, the first $n-d$ symbols are distinct (and hence in the order $1,2,\dots,n-d$).
\end{lemma}
\begin{proof}
Assume that there exists some symbol $x$ both of whose occurrences are within the first $n-d$ entries.  Thus the first occurrence of $n$ must be at position at least $n+1$, so the distance between $x$ and $n$ is at least $d+1$, a contradiction.
\end{proof}

\begin{lemma}
\label{closeones}
Let $d<n/2$ be the diameter of a double-$n$ string, and let $r_i$ be the number of symbols both of whose occurrences are within $d$ of either occurrence of $i$ for $1 \leq i \leq d+1$.  Then $r_i \leq 3d+i-n$.
\end{lemma}

\begin{proof} As $d <n/2$, Lemma \ref{distinctness} gives that the first $d+1$ symbols of such a double-$n$ string are $1, 2, \dots,d+1$.  For $1 \leq i \leq d+1$, there are only $i-1$ symbols before the first occurrence of $i$, so there are at most $3d+i-1$ symbols adjacent to $i$, of which $r_i$ of them are repeats.  Hence $n-1 \leq 3d+i-1-r_i$.
\end{proof}

\begin{cor}
\label{cordistinct}
For $1 \leq i \leq d+1$, and $d < n/2$, at most $3d+i-n$ of the symbols $1, 2,\dots, \hat{i}, \ldots ,d+1$,  are within $d$ of the second occurrence of $i$. (Here $\hat{i}$ means omit $i$).
\end{cor}
\begin{proof}
This follows directly from Lemma~\ref{closeones} since each of the symbols $1,\ldots, \hat{i},$ $\ldots, d+1$ occurs within $d$ of the first occurrence of $i$.
\end{proof}
We are now ready to prove the lower bound.

\begin{theorem}
\label{klowerbound}
Let $r$ be a positive integer.  Then $\delta(23r) \geq 8r$.  Thus the asymptotic approval ratio for double-$n$ strings is bounded below by $8/23$.
\end{theorem}
\begin{proof}
Let $n = 23r$ and let $S$ be a double-$n$ string with diameter $d$. Suppose $d < 8r$. Since $d$ is an integer we have $d \leq 8r - 1$. 
Note that $d \geq \delta(23r) \geq \delta(21r + 1) \geq 7r$ by Lemma \ref{lem:3r+1}. 

By Lemma \ref{distinctness} the first $n-d \geq 23r -8r + 1 =15r + 1$ symbols in $S$ are distinct (and in order). Now since $d < 8r$ the first occurrence of the symbol labeled $15r+1$ is not within $d$ of the first occurrence of $i$ for $1 \leq i \leq 7r+1$. Thus for any such $i$ we must have the second occurrence of $i$ occurring in one of three sets of positions, namely the block $B_1$ of length $d$ following the first occurrence of $15r+1$, the block $B_2$ of length $d$ ahead of the second occurrence of $15r+1$, or the block $B_3$ of length $d$ following the second occurrence of $15r+1$.  These blocks are illustrated in Figure~\ref{klawelowerbounddiagram}. Let $k_j$ be the number of symbols in $1 \leq i \leq 7r+1$ with their second occurrence in block $B_j$. From the preceding observation we have $k_1 + k_2 + k_3 \geq 7r + 1$ (conceivably such a second occurrence of $i$ could be in both $B_1$ and $B_2$ if they overlap).  

Note that any pair of symbols in $B_j$ lie within $d$ of each other. Suppose without loss of generality that the second occurrence of 1 lies in $B_1$. For any $i$ with $1 < i  \leq 7r+1$ with the second occurrence of $i$ in $B_1$, both occurrences of $i$ lie within $d$ of an occurrence of 1, since $d \geq 7r$. By Corollary \ref{cordistinct}, the number of such $i$ is at most 
$$3d + 1 - n  \leq 3(8r-1) + 1 - 23r = r - 2,$$ giving $k_1 \leq 1 + r - 2 = r - 1$. 

Let $x$ be the minimal number such that the second occurrence of $x$ is not in $B_1$. Then $x \leq r$ since $k_1 \leq r - 1$. Without loss of generality suppose the second occurrence of $x$ is in $B_2$. Again, by Corollary \ref{cordistinct}  there are at most
\[3d+r-n \leq 3(8r-1)+r-23r = 2r-3\]
symbols $i$ with $1 \leq i \leq 7r+1$ other than $x$ in $B_2$, so $k_2 \leq 2r-2$.

Similarly, let $y$ be the smallest symbol (in value) whose second occurrence is in $B_3$ (i.e., is not in $B_1$ or $B_2$).  There are at most $k_1+k_2$ symbols in $B_1 \cup B_2$, so $y \leq 3r-2$.  Using Corollary \ref{cordistinct} one last time, we see that there are at most $3d+(3r-2)-n\leq 3(8r-1)+(3r-2)-23r = 4r-5$ symbols $i \neq y$ with $1 \leq i \leq 7r+1$ in $B_3$, so $k_3 \leq 4r-4$.  However, this is a contradiction: we needed $k_1+k_2+k_3 \geq 7r+1$, but
\[k_1+k_2+k_3 \leq (r-1)+(2r-2)+(4r-4) = 7r-7.\]
Therefore we could not have $d<8r$, proving the theorem.
\end{proof}

A general argument showing $\delta(br)\geq ar$, for large $r$, leads to the inequalities $b<3a$ and $23a\leq 8b$. Thus the lower bound of Theorem \ref{klowerbound} is the best possible asymptotic bound using this argument.  
Now we turn to the upper bound. 

\begin{figure}
\centering
\includegraphics[width=4in]{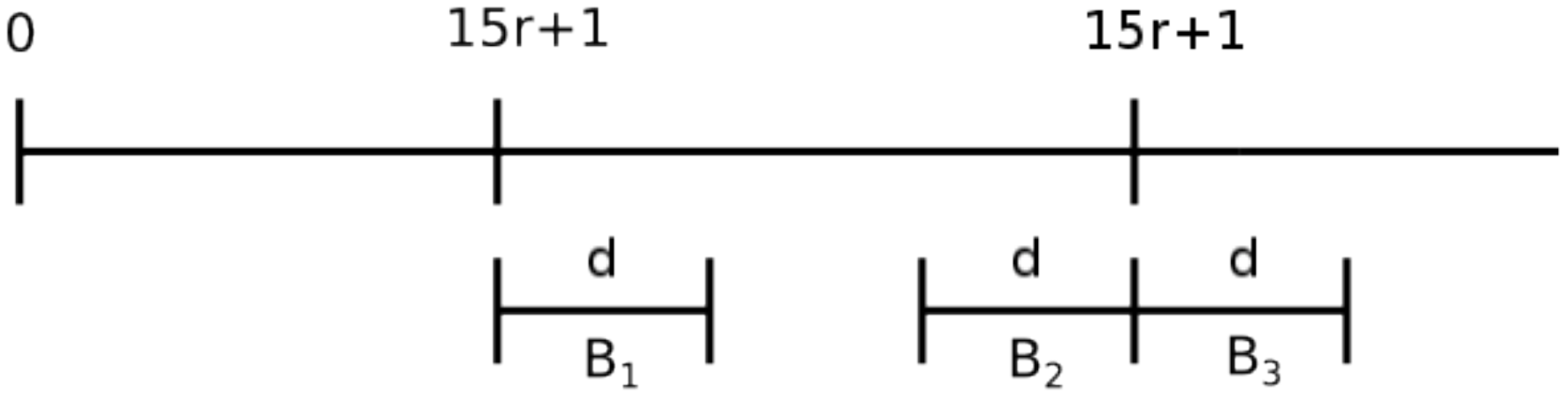}
\caption{$B_1$, $B_2$, and $B_3$ (in the case that they are disjoint).}
\label{klawelowerbounddiagram}
\end{figure}

\begin{theorem}
\label{klaweupperbound}
For any $n>0$, there exists a double-$n$ string with diameter $d \leq 5\ceil{\frac{n}{13}}-1$.  Hence the asymptotic approval ratio for double-$n$ strings is bounded above by $5/13$.

\end{theorem}
\begin{proof}
Note that the double-13 string
\[1,2,3,4,5,6,7,8,9,10,1,11,6,12,13, 5, 4, 7, 11,10, 9, 2, 3,13,12, 8\]
has diameter 4, meaning that any two symbols in it appear somewhere in the double-13 string separated by no more than three other elements.  This yields a general construction for double-$n$ strings of any length.  Let $k=\ceil{\frac{n}{13}}$.  Then replacing each symbol $i$ in the above string with the substring 
\[k(i-1)+1, k(i-1)+2,\dots, ki,\]
and removing any symbols in the resulting string that are greater than $n$, yields a double-$n$ string.  An example of this string for $n=34$ ($k=3$) is shown in Figure~\ref{klaweupperboundexample}.  Because the diameter of the above double-13 string is 4, any two symbols $1 \leq i<j \leq n$ are within substrings that are separated by at most three substrings of length $k$.  Also, $i$ and $j$ are at worst on the far ends of their substrings, giving a maximum total distance between $i$ and $j$ in the new string of\\
\[3\ceil{\frac{n}{13}}+\left(2\ceil{\frac{n}{13}}-1\right) =5\ceil{\frac{n}{13}}-1. \qedhere \]
\end{proof}

\begin{figure}
\centering
{\fontsize{10}{8}
\[(1,2,3)(4,5,6)(7,8,9)(10,11, 12)(13,14, 15)(16,17, 18)(19,20, 21)(22,23, 24)\]
\[(25,26, 27)(28,29, 30)(1,2, 3)(31,32,33)(16,17,18)(34)(13,14,15)(10,11,12)\]
\[(19,20, 21)(31,32, 33)(28,29, 30)(25,26, 27)(4,5, 6)(7,8, 9)(34)(22,23, 24)\]}
\caption{A double-34 string with diameter $\leq 14$ constructed as in Theorem~\ref{klaweupperbound}.  Symbols are grouped together by parentheses to elucidate its construction.  Some groupings have fewer than three elements since symbols larger than 34 in value are removed.  Empty groupings are also omitted.}
\label{klaweupperboundexample}

\end{figure}

\section{A double-interval society lower bound}
In the previous section we considered double-$n$ strings as examples of societies with low approval ratios.  These examples  give upper bounds for the minimal guaranteed approval ratio for any society.  In this section, we give lower bounds for the minimal guaranteed approval ratio, by considering how general pairwise-intersecting double-interval societies force conditions on the number of intervals that can intersect a given interval at its endpoints.  This approach largely ignores the geometry of the approval sets and considers only combinatorial constraints.  

\begin{theorem}
\label{lowerbound}
The approval number $a(S)$ of any $n$-voter society $S$ satisfies
\begin{equation} 
\label{nicebound} 
a(S) \geq \left\lceil 2n+\frac{1}{2}-\sqrt{3n^2-n+\frac{1}{4}} \right\rceil .
\end{equation}
Then the approval ratio satisfies
\begin{equation}
\label{niceapproval}
\frac{a(S)}{n}  \geq 2-\sqrt{3} + \frac{3+\sqrt{3}}{6n} - \frac{\sqrt{3}}{24n^{2}} 
\approx 0.268 + \frac{0.789}{n} -  \frac{1.732}{24n^{2}}.
\end{equation}
Alternatively, the size $n$ of a society achieving a given approval number $a(S)$ is bounded above by
\begin{equation}
\label{nicen}
n \leq \left\lfloor 2a(S) - \frac32 + \sqrt{3(a(S))^2 - 5 a(S) + \frac94} \right\rfloor.
\end{equation}
\end{theorem}

\begin{proof}
Let $A_i$ and $A_i'$ represent the left and right intervals, respectively, of voter $i$'s approval set in the $n$-voter society $S$.  Without loss of generality we may assume no two interval endpoints coincide.
For any interval $I$, 
define numbers $L(I)$, $R(I)$, $B(I)$, and $C(I)$ to keep track of the number of other intervals that intersect $I$ in various ways.  
Let $L(I)$ count the number of other intervals that, of two endpoints of $I$, contain only the \emph{left} endpoint.
Let $R(I)$ count the number of other intervals that, of two endpoints of $I$, contain only the \emph{right} endpoint.
Let $B(I)$ count the number of other intervals that contain \emph{both} endpoints of $I$.
Let $C(I)$ count the number of other intervals that intersect $I$ but contain \emph{neither} endpoint of $I$, and are hence in the ``center'' of $I$.

For example, in Figure~\ref{examplesociety}, we see that $L(A')=2$, $R(A')=0$, $C(A')=3$, and $B(A')=0$.  Also $L(C')=0$, $R(C')=1$, $C(C')=0$, and $B(C')=1$.  Since each set must intersect all $n-1$ other sets,
$$
L(A_i)+L(A_i')+R(A_i)+R(A_i')+C(A_i)+C(A_i')+B(A_i)+B(A_i') \geq n-1.
$$ 
Then clearly
\begin{eqnarray}
\sumi \left[ L(A_i)+L(A_i')+R(A_i)+R(A_i')+C(A_i')+C(A_i)+B(A_i)+B(A_i') \right] \nonumber \\
\geq n(n-1). 
\label{sumbound}
\end{eqnarray}

Note that an interval $J$ covers both endpoints of another interval $I$ and contributes $1$ to the count $B(I)$ exactly when $I$ is the in the center of $J$ and contributes $1$ to the count $C(J)$.  
This implies: 
\begin{equation}
\label{BC} 
\sumi \left[ B(A_i)+B(A_i') \right] = \sumi \left [ C(A_i)+C(A_i') \right].
\end{equation}

Notice that given an approval number $a(S)$, each interval may have at most $a(S)-1$ other sets intersecting its left endpoint. This gives an initial bound
\[\sumi \left[ L(A_i)+L(A_i')+B(A_i)+B(A_i') \right] \leq 2n(a(S)-1).\]
and similarly, considering right endpoints:  
\[\sumi \left[ R(A_i)+R(A_i')+B(A_i)+B(A_i') \right] \leq 2n(a(S)-1).\]
However, if the $2n$ intervals are ordered by the left endpoint, then the $k$th interval under this ordering from left to right can have at most $k-1$ intervals intersecting its left endpoint, not $a(S)-1$.  Thus we need to adjust the formulas above, to obtain:
\[\sumi \left[ L(A_i)+L(A_i')+B(A_i)+B(A_i') \right] \leq 2n(a(S)-1)-\frac{a(S)(a(S)-1)}{2}, \]
\[\sumi \left[ R(A_i)+R(A_i')+B(A_i)+B(A_i') \right] \leq 2n(a(S)-1)-\frac{a(S)(a(S)-1)}{2}.\]
Adding these equations and applying equation~(\ref{BC}) yields\\
\begin{eqnarray}
\sumi \left[ L(A_i)+L(A_i')+R(A_i)+R(A_i')+C(A_i')+C(A_i)+B(A_i)+B(A_i') \right] \nonumber \\
\leq 4n(a(S)-1)-a(S)(a(S)-1). \nonumber
\end{eqnarray}
So by equation~(\ref{sumbound}), we see\\
\[ \label{niceeqnbound}(4n-a(S))(a(S)-1) \geq n(n-1).\]
Solving this quadratic inequality for $a(S)$, and rounding up to the nearest integer gives the conclusion (\ref{nicebound}).  Using $(1-x)^{1/2} \leq 1-(1/2)x$ gives conclusion (\ref{niceapproval}). Solving the quadratic inequality for $n$ and rounding down gives the conclusion (\ref{nicen}).
\end{proof}

Values of $a(S)$ and the corresponding bounds on $n$ and the approval ratio derived from equation~(\ref{nicen}) are given in Table~\ref{boundresults}.

\begin{table}
  \centering
  \begin{tabular}{c|cc|cc}
  \hline
  $a(S)$ & $n$ & Approval & Observed & Observed
  \\
  & & Ratio&  $n$ & Approval Ratio\\
  \hline
  $2$ & $\leq 4$ & $\geq 0.500$ & 4 & 0.500\\
  $3$ & $\leq 8$ & $\geq 0.375$ & 8 & 0.375\\
  $4$ & $\leq 12$ & $\geq 0.333$ & 12 & 0.333\\
  $5$ & $\leq 15$ & $\geq 0.333$ & 15 & 0.333\\
  $6$ & $\leq 19$ & $\geq 0.316$ & 18 & 0.333\\
  $7$ & $\leq 23$ & $\geq 0.304$ & 21 & 0.333\\
  $8$ & $\leq 26$ & $\geq 0.308$ & 24 & 0.333\\
  $9$ & $\leq 30$ & $\geq 0.300$ & 27 & 0.333\\
  $10$ & $\leq 34 $ & $\geq 0.294$ & 30 & 0.333\\
  $11$ & $\leq 38$ & $\geq 0.289$ & 32 & 0.344\\
  $12$ & $\leq 41 $ & $\geq 0.293$ & 35 & 0.343\\
\hline
  \end{tabular}
  \bigskip
  \caption{On the left, this table shows for a given approval number the largest $n$ that is given by inequality (\ref{nicen}) as well as the resulting bound on the approval ratio derived from inequality (\ref{nicebound}).  On the right, this table shows, for a given approval number, known examples of the largest $n$ that has this approval number and the observed approval ratio in that case, obtained by a modification of a double-$n$ string construction.}
  \label{boundresults}
\end{table}

\section{Modifying double-$n$ string societies}

In this section we give an example of a double-interval society with an approval ratio lower than the bound given by Theorem \ref{klowerbound}, thus showing that
double-$n$ strings do not always provide examples of societies with minimal approval ratios.  
We will require a new notation, called the \emph{endpoint representation} of a society.  We will encode a society as a sequence of symbols (corresponding to the approval sets) representing the order of the endpoints of all the approval sets, each prefixed by a ``+'' or a ``$-$' to denote a left or right endpoint respectively.  For example, the society in Figure~\ref{examplesociety} is represented as
\[+A+C+B-A+D-C+A-B-D+C+B-C+D-B-D-A.\]

\begin{prop}
\label{counterexample}
There exists a society of size $n=8$ with approval number 3.  Hence there exist $n$ for which double-$n$ strings do not produce the lowest possible approval numbers. 
\end{prop}

\begin{figure}
\centering
\setlength{\unitlength}{11pt}
\begin{picture}(32,20)
\linethickness{2pt}

\mysetb{1}{19.8}{3}{15}{10.2}{7}{darkred}
\mysetb{2}{18.6}{4}{23}{5.4}{8}{darkblue}
\mysetb{3}{17.4}{5}{25}{4.2}{5}{darkyellow}
\mysetb{7}{15}{3}{19}{7.8}{5}{darkgreen}
\mysetb{5}{16.2}{7}{17}{9}{1}{orange}
\mysetb{11}{12.6}{3}{21}{6.6}{5}{violet}
\mysetb{9}{13.8}{7}{27}{3}{1}{teal}
\mysetb{13}{11.4}{7}{29}{1.8}{3}{brown}

\mylabel{2.25}{20}{18.25}{10.4}{A}
\mylabel{3.75}{18.8}{26.75}{5.6}{B}
\mylabel{5.25}{17.6}{27.25}{4.4}{C}
\mylabel{8.25}{15.2}{21.25}{8}{D}
\mylabel{8.25}{16.4}{17.25}{9.2}{E}
\mylabel{12.25}{12.8}{23.25}{6.8}{F}
\mylabel{12.25}{14}{27.25}{3.2}{G}
\mylabel{16.25}{11.6}{30.25}{2}{H}

\end{picture}
\caption{A society of size 8 with approval number 3.}
\label{3-8counterexample}
\end{figure}
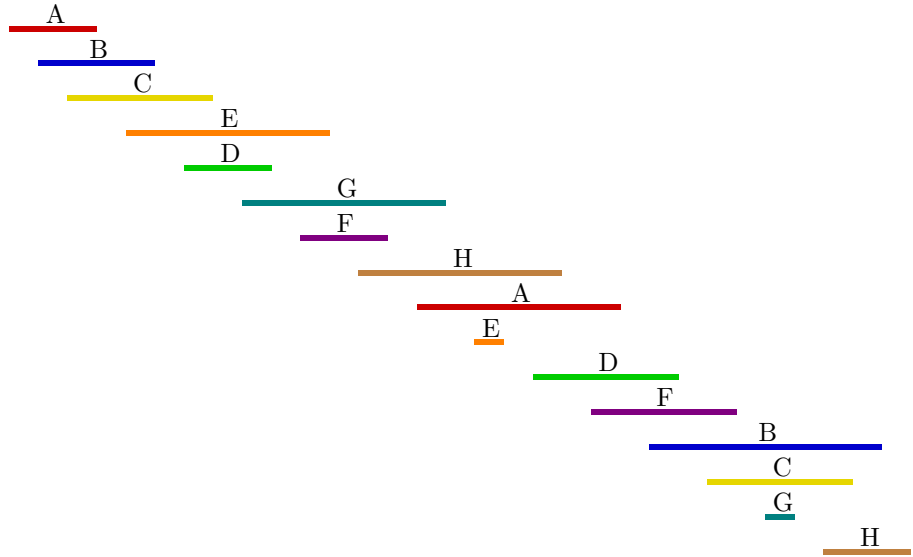

The society shown in Figure~\ref{3-8counterexample} provides such a society.  This example was  
 derived from the double-8 string
  \[ABCDEFGHEADFCGBH.\]
 If each interval in the string overlaps two intervals on each side, this arrangement 
 is missing the adjacencies $AG$, $BE$, $BF$, $CA$, $DG$, $DG$ and $CH$ and has duplicate adjacencies $BC$, $CD$, $DC$, $DE$, $DF$, $EG$, $FG$, and $GH$.  
By doing a series of moves that interchanges endpoints in such a way as to introduce missing adjacencies (at the expense of duplicate adjacencies) without increasing the approval number, we arrive at the society
\begin{eqnarray*}
&&+A+B+C-A+E-B+D-C+G-D+F-E+H-F+A-G\\
&&+E-E+D-H+F-A+B-D+C-F+G-G+H-C-B-H.
\end{eqnarray*}

We note that an example like this with $n=8$ and $a(S)=3$ cannot be achieved by a double-$n$ string since the first symbol in a  double-$n$ string with diameter $d$ is adjacent to at most $3d$ other symbols. Thus, as in Lemma \ref{lem:3r+1}, we have
 $n \leq 3d+1 = 3(a(S)-1) +1$, and so the approval number of a double-8 string must be at least 4.
 
 It is not clear how to systematically interchange endpoints to achieve all possible adjacencies. However, an algorithm which aims at making ``smart'' swaps produced societies with approval ratios given in Table \ref{boundresults}. A description of the algorithm can be found in \cite{Scott}.
The results of the hill-climbing algorithm in Figure~\ref{hillclimbingoutput} suggest that the asymptotic approval ratio should be $1/3$.

\begin{figure}  
\begin{verbatim}
a(S)=3, n=8,  AR=0.375: +A+B+F-F+G-A+F-B+C-C+D-G+E-D+H-F+G-G+A-E+D-H
                        +C-A+B-D+E-E+H-H-B-C   
a(S)=4, n=12, AR=0.333: +A+B+C+F-C+H-B+L-L+G-G+I-F+D-A+C-H+L-D+E-E+J
                        -I+G-J+K-C+B-L+I-I+D-K+E-G+J-B+F-D+K-F+A-A+H
                        -K-E-H-J
a(S)=5, n=15, AR=0.333: +A+B+C+D+E-C+G-A+O-G+F-D+K-F+J-J+N-E+C-B+A-O
                        +D-K+H-H+M-N+J-M+L-L+I-D+F-A+G-C+N-I+L-N+H-J
                        +M-F+K-G+I-K+B-B+E-E+O-H-M-I-L-O
a(S)=6, n=18, AR=0.333: +A+B+C+D+E+G-A+M-C+K-G+Q-Q+P-P+O-O+J-D+F-E+A
                        -B+C-F+P-K+I-I+R-M+O-R+H-H+L-J+Q-L+N-A+G-C+J
                        -J+F-P+I-O+R-N+H-Q+L-G+D-F+N-D+K-K+B-B+E-R+M
                        -N-M-E-H-L-I
a(S)=7, n=21, AR=0.333: +A+B+C+D+E+F+I-A+K-K+N-N+R-R+G-E+T-B+J-T+U-G
                        +Q-D+B-I+P-Q+M-C+H-F+L-J+S-U+O-B+D-D+C-C+I-I
                        +F-F+G-P+Q-M+R-L+J-S+N-O+K-H+A-G+E-J+T-Q+P-P
                        +M-M+U-R+S-S+O-U+L-N+H-T-H-K-L-E-O-A
a(S)=8, n=24, AR=0.333: +A+B+C+D+E+F+G+L-G+N-C+O-F+Q-O+M-N+T-B+H-H+K
                        -Q+I-I+U-U+X-D+J-L+F-A+O-E+C-J+G-M+P-T+U-P+H
                        -K+W-X+I-W+R-R+V-V+S-O+Q-C+J-F+N-G+B-U+X-S+P
                        -H+R-I+V-X+W-Q+K-B+D-J+T-N+S-D+L-T+M-K+A-L+E
                        -P-A-W-S-E-M-V-R
 \end{verbatim}
\caption{Output pairwise-intersection double-interval societies with given sizes and approval ratios found by a heuristic 
algorithm. Here $AR$ denotes the approval ratio.}
\label{hillclimbingoutput}
\end{figure}

\bigskip

\section{Conclusion and Open Questions}

We have studied pairwise-intersecting double-interval societies, and determined bounds for the minimum guaranteed approval ratio for such societies.  Such questions naturally motivated the study of double-$n$ strings, which represent certain special double-interval societies with low approval ratios.  Although these do not necessarily provide the smallest such ratios, all of the known examples that provide smaller ratios come from modifying the double-$n$ string construction.

There are numerous open questions.  

\begin{itemize}
\item For double-$n$ strings, is there a systematic way to construct strings of the smallest diameter?
\item
Beyond double-$n$ strings, is there a better general construction that yields societies with the lowest approval ratios?  
\item
With double-$n$ strings, we currently have $\Delta$ bounded by $0.348 \leq \Delta \leq 0.385$. Can we tighten the bounds on $\Delta$?
\item
What results can be obtained for triple-interval societies?
\item
What about higher-dimensional approval sets?  What can be said if each voter's approval set consists of two convex sets in the plane?
\end{itemize}

Finally, we end with our initial conjecture, which now has more evidence as support. 
\begin{conjecture}
For all pairwise-intersecting double-interval societies $S$, the approval ratio $$\frac{a(S)}{n} \geq \frac13.$$
\end{conjecture}

\bibliographystyle{amsplain}

\end{document}